\documentclass{article}
\usepackage{amsmath,amssymb,amsthm,graphicx,psfrag}
\usepackage{enumerate,subfigure}
\usepackage{pstricks,pst-node,multido}

\usepackage[T1]{fontenc}
\usepackage[latin1]{inputenc}
\usepackage{psfrag}

\newtheorem{thm}{Theorem}[section]
\newtheorem{lma}[thm]{Lemma}

\newtheorem{conj}[thm]{Conjecture}

\theoremstyle{definition}
   \newtheorem{df}[thm]{Definition}
   \newtheorem{ex}[thm]{Example}

\theoremstyle{remark}

\numberwithin{thm}{section}
\numberwithin{figure}{section}
\numberwithin{equation}{section}

\DeclareMathOperator{\PPD}{PPD}
\DeclareMathOperator{\Per}{Per}
\DeclareMathOperator{\TPD}{TPD}
\DeclareMathOperator{\Graph}{Graph}

\newcommand{\F}{\mathbb{F}}

\title{Towards a bifurcation theory for perturbed monomial dynamical systems modulo a prime}
\author{Marcus Nilsson}

\begin{document}
\maketitle
\begin{abstract}
	We investigate perturbed monomial dynamical system over $\F_p$ given by iterations of
	$x\mapsto x^n+c\bmod{p}$, where $c\in \F_p$. Instead of study the systems one at a time we study all of them at the same time. The complex distibution of periodic points is visualized in the so called Periodic Point Diagram, which can be seen as a discrete version of the classical Bifurcation Diagram. We also prove some general results about the distribution of periodic points. We end the article with a conjecture about the total number of periodic points. 
\end{abstract}
\section{Introduction}
We consider the dynamical system given by iterations of
\begin{equation}
	h_c(x)=x^n+c
\end{equation}
over the prime field $\F_p$, the set $\{0,1,\ldots, p-1\}$ with addition and multiplication modulo $p$. We let $n\geq 2$ be an integer and $c\in \{0,1,\ldots, p-1\}$. By $h_c^r$ we mean the $r$-fold composition of $h_c$.
The dynamics of $h_c(x)$ changes dramatically with the value of the parameter $c$. Classically we have a bifurcation if there is a sudden change in the dynamics at a certain parameter value, see for example \cite{Holmgren:1996}. If we use the classical definition of a bifurcation point we can say that we have a bifurcation for every value of $c$. When describing the dynamics we will concentrate on the number of periodic points. A point $a$ is said to be a periodic point if $h_c^r(a)=a$ for some positive $r$, the smallest such $r$ is called the period of $a$. We say that $a$ is an $r$-periodic point. Since we are dealing with a finite set (all residues modulo a prime $p$) every point will after a number of iterations be mapped onto a periodic point. The set of points can therefore be partitioned into two sets, the periodic points and the preperiodic points that will eventually be mapped onto periodic points.  

We will use the notation $\F_p^*$ for the multiplicative group of $\F_p$. 
We will visualize the dynamics of $h_c$ on $\F_p$ by a directed graph $\Graph(V,E)$, where the the vertex set is $V=\F_p$ and the edge set $E$, is the set $\{(x,h_c(x)),x\in V\}$.
\begin{ex}
	Let us study the dynamics of the $h_c(x)=x^2+c$ on $\F_7$ for all seven possible choices of $c$. We iterate the map $x\mapsto x^2+c\bmod{7}$. The result is shown in Figure \ref{dynamicsintro}. Here we can see that the dynamics changes a lot if we change the value of $c$.
	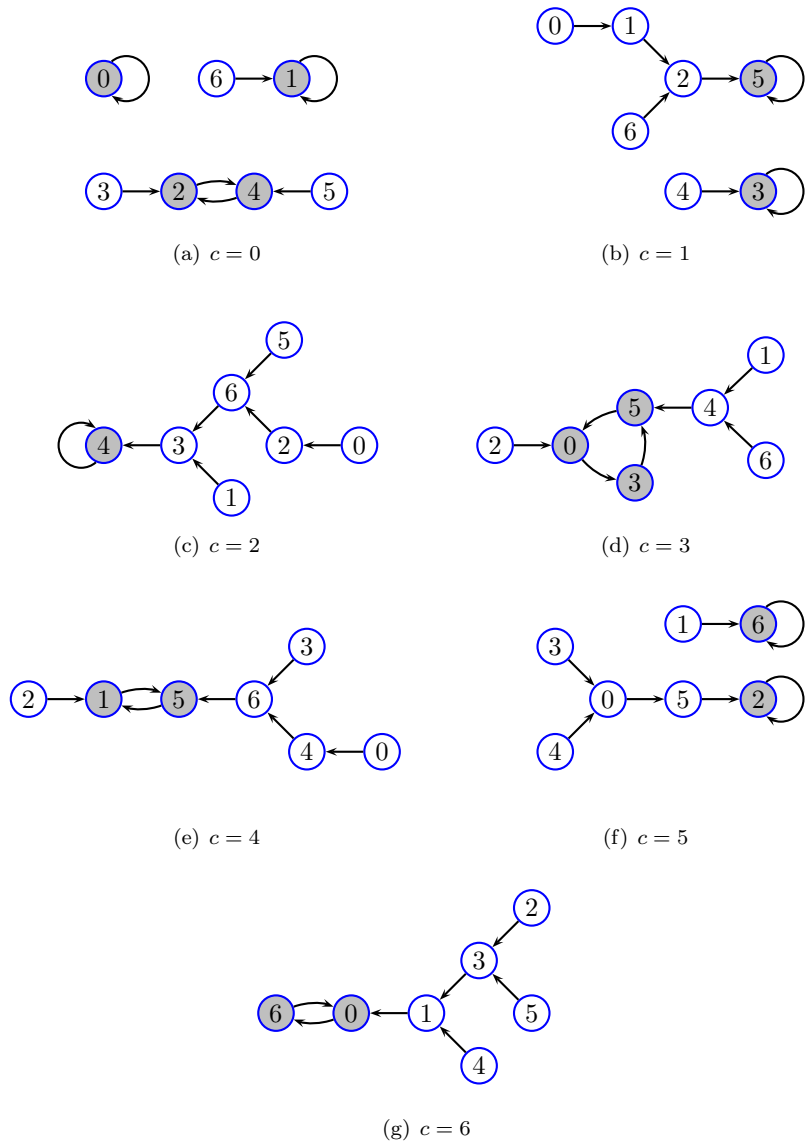
\begin{figure}
		\centering
		\subfigure[$c=0$]{
		\begin{pspicture}(0,0)(5,3)
			\psset{linecolor=blue}
			\Cnode[fillstyle=solid,fillcolor=lightgray](1,2){0}\rput(1,2){$0$}
			\Cnode(2.5,2){6}\rput(2.5,2){$6$}
			\Cnode[fillstyle=solid,fillcolor=lightgray](3.5,2){1}\rput(3.5,2){$1$}
			\Cnode(1,0.5){3}\rput(1,0.5){$3$}
			\Cnode[fillstyle=solid,fillcolor=lightgray](2,0.5){2}\rput(2,0.5){$2$}
			\Cnode(4,0.5){5}\rput(4,0.5){$5$}
			\Cnode[fillstyle=solid,fillcolor=lightgray](3,0.5){4}\rput(3,0.5){$4$}
			\psset{linecolor=black}
			\ncline{->}{6}{1}
			\ncline{->}{3}{2}
			\ncline{->}{5}{4}
			\ncarc[arcangle=20]{->}{2}{4}
			\ncarc[arcangle=20]{->}{4}{2}
			\nccircle[angleA=-90]{<-}{0}{0.3cm}
			\nccircle[angleA=-90]{<-}{1}{0.3cm}
		\end{pspicture}
		}
		\quad
		\subfigure[$c=1$]{
		\begin{pspicture}(0,0)(5,3)
		\psset{linecolor=blue}
			\Cnode(1.3,2.7){0}\rput(1.3,2.7){$0$}
			\Cnode(2.3,2.7){1}\rput(2.3,2.7){$1$}
			\Cnode(3,2){2}\rput(3,2){$2$}
			\Cnode[fillstyle=solid,fillcolor=lightgray](4,2){5}\rput(4,2){$5$}
			\Cnode(2.3,1.3){6}\rput(2.3,1.3){$6$}
			\Cnode(3,0.5){4}\rput(3,0.5){$4$}
			\Cnode[fillstyle=solid,fillcolor=lightgray](4,0.5){3}\rput(4,0.5){$3$}
				\psset{linecolor=black}
			\ncline{->}{0}{1}
			\ncline{->}{1}{2}
			\ncline{->}{6}{2}
			\ncline{->}{2}{5}
			\ncline{->}{4}{3}
			\nccircle[angleA=-90]{<-}{5}{0.3cm}
			\nccircle[angleA=-90]{<-}{3}{0.3cm}
		\end{pspicture}
		}\\
		\subfigure[$c=2$]{
		\begin{pspicture}(-0.5,0)(4.5,3)
			\psset{linecolor=blue}
			\Cnode[fillstyle=solid,fillcolor=lightgray](0.5,1){4}\rput(0.5,1){$4$}
			\Cnode(1.5,1){3}\rput(1.5,1){$3$}
			\Cnode(2.2,0.3){1}\rput(2.2,0.3){$1$}
			\Cnode(2.2,1.7){6}\rput(2.2,1.7){$6$}
			\Cnode(2.9,1){2}\rput(2.9,1){$2$}
			\Cnode(2.9,2.4){5}\rput(2.9,2.4){$5$}
			\Cnode(3.9,1){0}\rput(3.9,1){$0$}
			\psset{linecolor=black}
			\ncline{->}{0}{2}
			\ncline{->}{2}{6}
			\ncline{->}{5}{6}
			\ncline{->}{6}{3}
			\ncline{->}{1}{3}
			\ncline{->}{3}{4}
			\nccircle[angleA=90]{<-}{4}{0.3cm}
		\end{pspicture}
		}
		\quad
		\subfigure[$c=3$]{
		\begin{pspicture}(0,0)(5,3)
			\psset{linecolor=blue}
			\Cnode(0.5,1){2}\rput(0.5,1){$2$}
			\Cnode[fillstyle=solid,fillcolor=lightgray](1.5,1){0}\rput(1.5,1){$0$}
			\Cnode[fillstyle=solid,fillcolor=lightgray](2.36,0.5){3}\rput(2.36,0.5){$3$}
			\Cnode[fillstyle=solid,fillcolor=lightgray](2.36,1.5){5}\rput(2.36,1.5){$5$}
			\Cnode(3.36,1.5){4}\rput(3.36,1.5){$4$}
			\Cnode(4.1,0.8){6}\rput(4.1,0.8){$6$}
			\Cnode(4.1,2.2){1}\rput(4.1,2.2){$1$}
			\psset{linecolor=black}
			\ncline{->}{1}{4}
			\ncline{->}{6}{4}
			\ncline{->}{4}{5}
			\ncline{->}{2}{0}
			\ncarc[arcangle=-20]{->}{5}{0}
			\ncarc[arcangle=-20]{->}{0}{3}
			\ncarc[arcangle=-20]{->}{3}{5}
		\end{pspicture}
		}\\
		\subfigure[$c=4$]{
		\begin{pspicture}(0,0)(5,3)
			\psset{linecolor=blue}
			\Cnode(0,1.5){2}\rput(0,1.5){$2$}
			\Cnode[fillstyle=solid,fillcolor=lightgray](1,1.5){1}\rput(1,1.5){$1$}
			\Cnode[fillstyle=solid,fillcolor=lightgray](2,1.5){5}\rput(2,1.5){$5$}
			\Cnode(3,1.5){6}\rput(3,1.5){$6$}
			\Cnode(3.7,0.8){4}\rput(3.7,0.8){$4$}
			\Cnode(3.7,2.2){3}\rput(3.7,2.2){$3$}
			\Cnode(4.7,0.8){0}\rput(4.7,0.8){$0$}
			\psset{linecolor=black}
			\ncline{->}{2}{1}
			\ncline{->}{0}{4}
			\ncline{->}{4}{6}
			\ncline{->}{3}{6}
			\ncline{->}{6}{5}
			\ncarc[arcangle=20]{->}{1}{5}
			\ncarc[arcangle=20]{->}{5}{1}
		\end{pspicture}
		}
		\quad
		\subfigure[$c=5$]{
		\begin{pspicture}(0,0)(5,3)
		\psset{linecolor=blue}
			\Cnode(1.3,0.8){4}\rput(1.3,0.8){$4$}
			\Cnode(1.3,2.2){3}\rput(1.3,2.2){$3$}
			\Cnode(2,1.5){0}\rput(2,1.5){$0$}
			\Cnode(3,1.5){5}\rput(3,1.5){$5$}
			\Cnode[fillstyle=solid,fillcolor=lightgray](4,1.5){2}\rput(4,1.5){$2$}
			\Cnode(3,2.5){1}\rput(3,2.5){$1$}
			\Cnode[fillstyle=solid,fillcolor=lightgray](4,2.5){6}\rput(4,2.5){$6$}
			\psset{linecolor=black}
			\ncline{->}{4}{0}
			\ncline{->}{3}{0}
			\ncline{->}{0}{5}
			\ncline{->}{5}{2}
			\ncline{->}{1}{6}
			\nccircle[angleA=-90]{<-}{2}{0.3cm}
			\nccircle[angleA=-90]{<-}{6}{0.3cm}
		\end{pspicture}
		}\\
		\subfigure[$c=6$]{
		\begin{pspicture}(0,0)(5,3)
		\psset{linecolor=blue}
			\Cnode[fillstyle=solid,fillcolor=lightgray](0.5,1.2){0}\rput(0.5,1.2){$6$}
			\Cnode[fillstyle=solid,fillcolor=lightgray](1.5,1.2){6}\rput(1.5,1.2){$0$}
			\Cnode(2.5,1.2){1}\rput(2.5,1.2){$1$}
			\Cnode(3.2,0.5){4}\rput(3.2,0.5){$4$}
			\Cnode(3.2,1.9){3}\rput(3.2,1.9){$3$}
			\Cnode(3.9,1.2){5}\rput(3.9,1.2){$5$}
			\Cnode(3.9,2.6){2}\rput(3.9,2.6){$2$}
			\psset{linecolor=black}
			\ncline{->}{2}{3}
			\ncline{->}{5}{3}
			\ncline{->}{3}{1}
			\ncline{->}{4}{1}
			\ncline{->}{1}{6}
			\ncarc[arcangle=20]{->}{0}{6}
			\ncarc[arcangle=20]{->}{6}{0}
		\end{pspicture}
		}
	\label{dynamicsintro}
	\caption{The dynamics of the seven systems $x\mapsto x^2+c \bmod{7}$.}
	\end{figure}
\end{ex}
The systems $x\mapsto x^2+c$ on $\F_p$ have been studied before. In \cite{Rogers:1996} Rogers discribes the connected components of the directed graph of $x\mapsto x^2$. Rogers stated that for $c\ne 0$ the dynamics seems to be beyond description. In \cite{Vasiga:2004}
 Vasiga and Shallit succeeded in describing the dynamics of $x\mapsto x^2-2$. Also Gilbert et al in \cite{Gilbert:2001} contributed to the understanding of the dynamics of $x\mapsto x^2-2$. We can see this already in Figure \ref{dynamicsintro},  the graphs for $c=5$ and $c=0$ seems more symmetric than the other graphs.   In \cite{Khrennikov/Nilsson:2001} the number of cycles (periodic orbits) of $x\mapsto x^n \bmod{p}$ are calculated by using M\"obius inversion formula. Asymptotical behavior for the number of periodic points when $p\to \infty$ where investigated. The results from \cite{Khrennikov/Nilsson:2001} where extended in \cite{Nilsson:2003} and in \cite{Nilsson:2007}.  
 
In this article we will construct a bifurcation diagram analogously to the classical diagram describing for example the bifurcations of the logistic map $x\mapsto rx(1-x)$ over the real numbers. See for example \cite{Holmgren:1996} for details on this dynamical systems. In Figure \ref{logisticbifurcation} a bifurcation diagram for the logistic map is given. A variant of the Mathematica program suggested by \cite{Holmgren:1996} was used for producing the picture.
\begin{figure}
\centering
\includegraphics{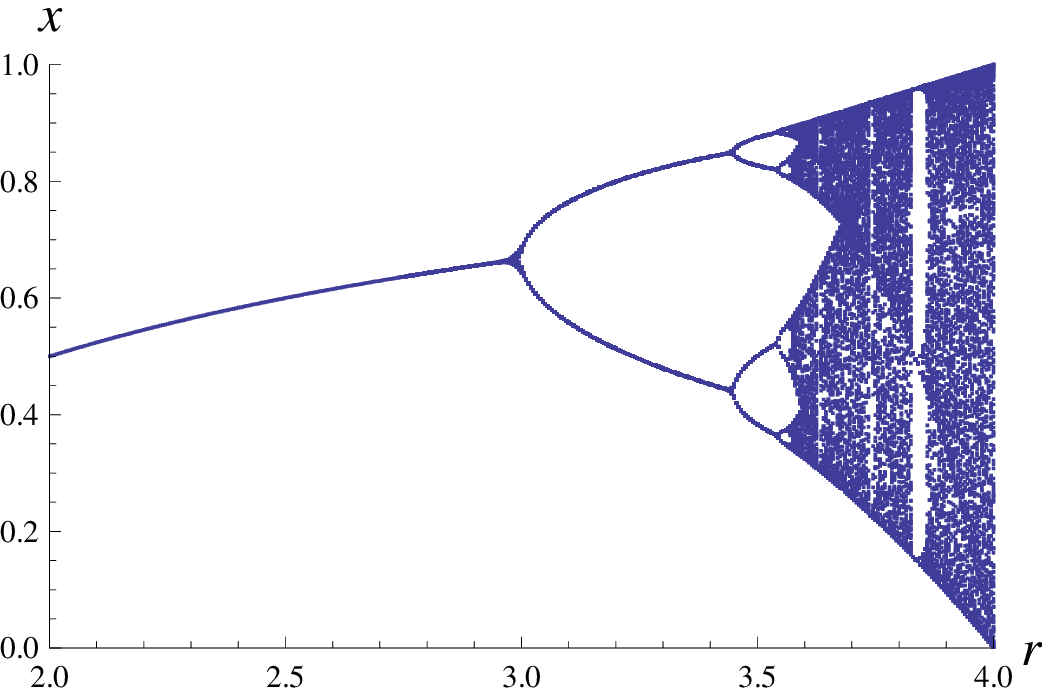}
	\caption{The bifurcation diagram of $x\mapsto rx(1-x)$ for $r\in [2,4]$. This is just iterations of the starting value $x=0.5$.}
	\label{logisticbifurcation}	
\end{figure}

This article is organized as follows: In Section \ref{sec:PPD} we introduce the Periodic Point Diagram ($\PPD$) and compare it with the classical bifurcation diagram. We also prove some general properties of the $\PPD$. Then in Section \ref{sec:quadratic} and Section \ref{sec:npower} we investigate the properties of $\PPD$ in more detail for $n=2$ and $n>2$, respectively. In Section \ref{sec:conjectures} we formulate a conjecture on the asymptotical behavior of the total number of periodic points of  set $\{x\mapsto x^n+c; c\in \F_p\}$ of dynamical systems when $p\to\infty$.  
 
\section{The periodic point diagram}\label{sec:PPD}
First of all there is an important difference between the bifurcation diagram we will construct for $x\mapsto x^n+c \bmod{p}$ and the classical one; our diagram is discrete. We will include every possible value of both $x$ and $c$.
We construct the diagram in the following way: On the horizontal axis we have all the possible values of $c$, $c\in\mathbb{F}_p$ and on the vertical axis we have the  possible values of $x$, $x\in \mathbb{F}_p$. For each $c$ we mark the $x$-values that are periodic points.
\begin{ex}
	Let us consider the dynamical system $x\mapsto x^2+c \bmod{7}$. In Figure \ref{first-ppd-7} the bifurcation diagram for this system is shown. It is hard to see any pattern from one $c$ to the next. However, there seems to be diagonal lines with no periodic points.
	\begin{figure}
	\centering
		\includegraphics[scale=0.8]{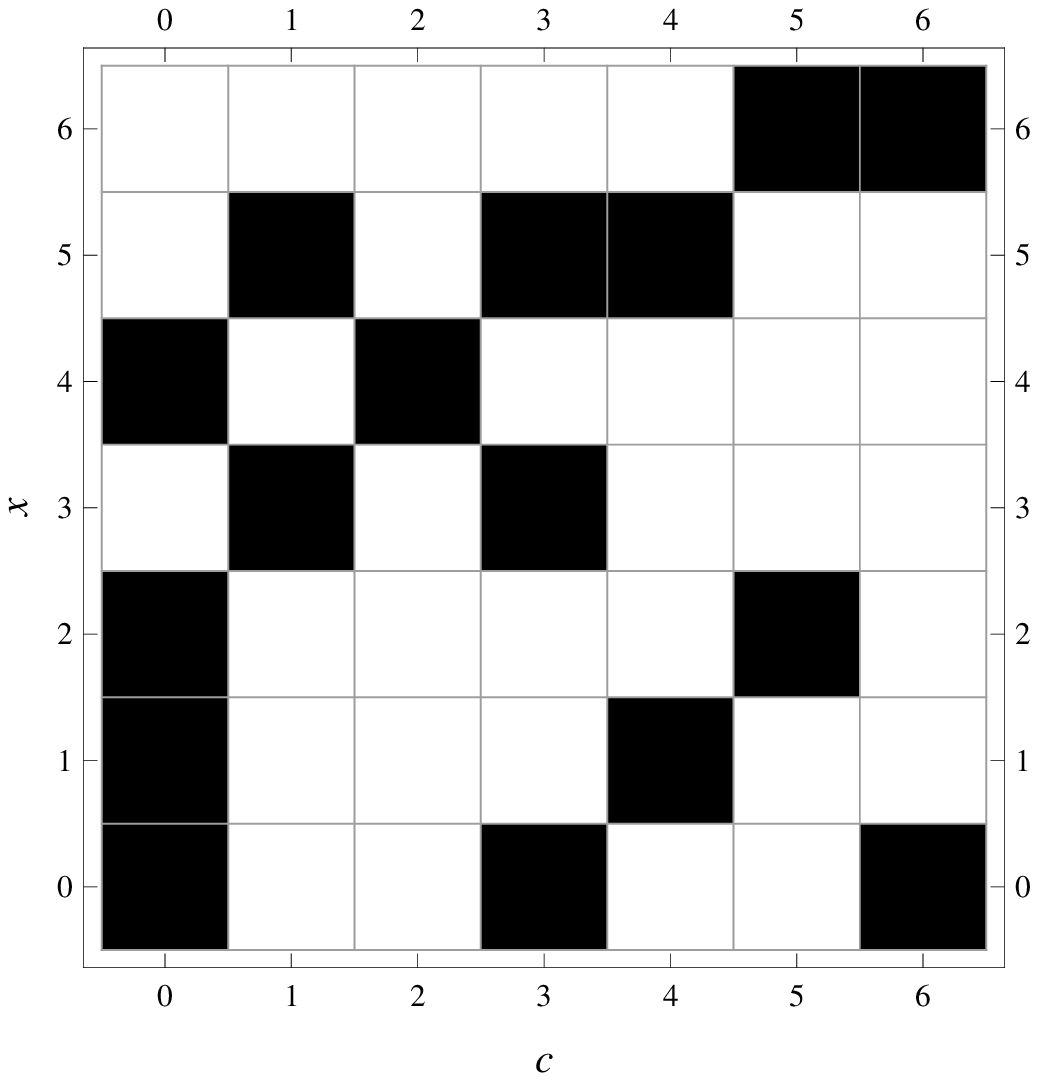}
		\caption{The bifurcation diagram of $x\mapsto x^2+c \bmod{7}$.}
		\label{first-ppd-7}
	\end{figure}
\end{ex}

The diagram in Figure \ref{first-ppd-7} was made by using a Mathematica program, that naively computes the periodic points for each $c$:
Iterate every point in $\mathbb{F}_p$, $p$ times then start saving the next $p$ iterates. Then take the union of all resulting points. In this way we are sure to have found all the periodic points and no preperiodic points.
Of course there a more effective algorithms that uses for example Floyd's cycle finding algorithm (Tortoise and Hare), see \cite{Floyd:1967,Knuth2:1969}, but creating the diagrams in this article the naive way is enough. 

We now make some formal definitions. First of all, instead of saying ''bifurcation diagram'' we call the diagram \emph{Periodic point diagram} ($\PPD$), it seems like a better name since we do not really have the same kind of bifurcation in our finite case as in the classical case. The PPD of the parametrized dynamical system given by iterations of the mapping $x\mapsto h_c(x)$ modulo $p$ is denoted by $\PPD(h_c,p)$. Let $a\in \F_p$.
By a diagonal line in the $\PPD(h_c,p)$ we mean all points $(i,j)$ in the diagram satisfying the equation $j=i+a\bmod{p}$. For $a=0$ we have the diagonal from the lower left corner to the upper right corner. All other diagonal lines are divided into two parts, see Figure \ref{diagonal-picture}.  
\begin{figure}
\centering
\includegraphics[scale=0.8]{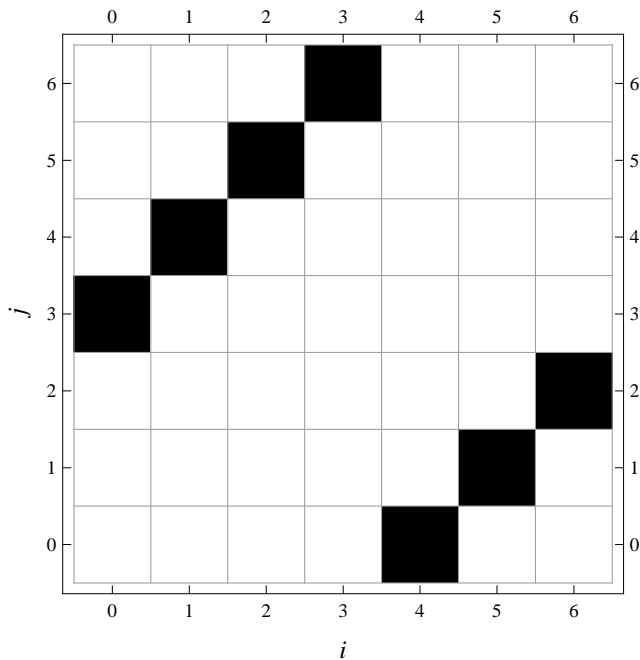}
\caption{The diagonal line $j=i+3 \bmod{7}$ in the $\PPD(h_c,7)$.}
\label{diagonal-picture}
\end{figure}
\begin{df}
	By a \emph{desert line} of $\PPD(h_c,p)$ we mean a diagonal line with no periodic points. 
\end{df}
\begin{ex}\label{ex:ppd71}
	Let us now consider the dynamical system $x\mapsto x^2+c\bmod{71}$. In Figure \ref{ppd-71} we have its $\PPD$. Here the desert lines are clearly visible.
	\begin{figure}
		\centering
		\includegraphics[scale=0.8]{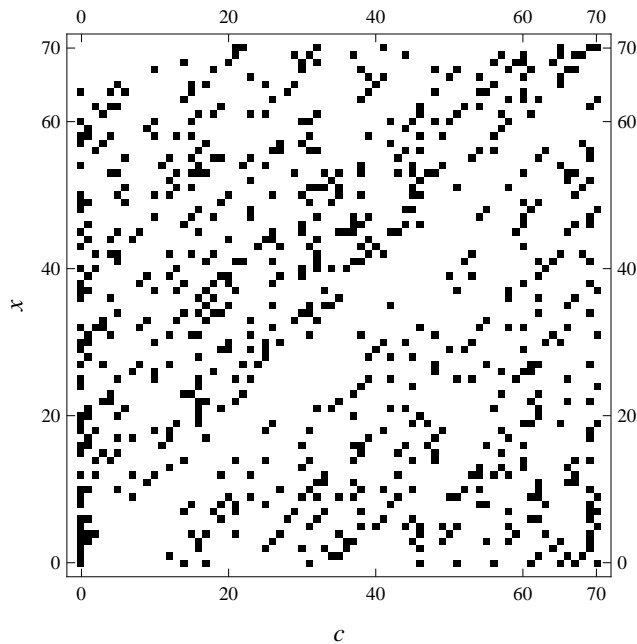}
		\caption{The bifurcation diagram of $x\mapsto x^2+c\bmod{71}$.}
		\label{ppd-71}
	\end{figure}
\end{ex}
Before we enter the discussion on how many desert lines there are, we take a look at some conditions for $h_c$ to be bijective.
\begin{lma}
	If $h_0$ is bijective then $h_c$ is bijective for all $c\in\{0,1,\ldots, p-1\}$.
\end{lma}
\begin{proof}
	Assume that the map $h_0(x)=x^n$ is bijective.
	Let $x,y\in \F_p$. If $x^n+c= y^n+c$ then
	$x^n = y^n $ and this implies $x = y$ since 
	$x\mapsto x^n \bmod{p}$ is bijective.
\end{proof}
In fact, we can generalize the lemma to: If $h_c$ is bijective for one $c$ it is bijective for all. 
Moreover, if the map $h_c$ is bijective then the dynamical system can't have any preperiodic points. Hence the dynamics has a pure cycle structure if and only if $h_c$ is bijective.
\begin{thm}\label{thm_bijective}
	The dynamics of $h_c(x)=x^n+c$ has a pure cycle structure if and only if
	$\gcd(n,p-1)=1$.
\end{thm}
\begin{proof}
	In this proof we will use results from elementary number theory. See for example
	\cite{Rosen:2010} for details on primitive roots and solutions of congruence equations.
	The theorem follows if we can prove that $x\mapsto x^n\bmod{p}$
 is bijective if and only if $\gcd(n,p-1)=1$. Let $\mu$ be a primitive root modulo $p$. Let $x,y\in \F_p^*$ then there are $i,j\in \{0,1,\ldots,p-2\}$ such that
 $x=\mu^i$ and $y=\mu^j$. We have $x^n = y^n$ if and only if $ni\equiv nj\pmod{p-1}$. This equation has the unique solution $i\equiv j\pmod{p-1}$ (hence $i=j$ since they are both in $\{0,1,\ldots,p-2\}$) if and only if $\gcd(n,p-1)=1$.  Also note that $x^n = 0$ in $\F_p$ if and only if $x=0$.
 \end{proof}
In the next section we will prove that there are exacly $(p-1)/2$ desert lines in $\PPD(x^2+c,p)$ for $p\geq 3$.

\section{The quadratic case}\label{sec:quadratic}
In  this section we look at the $\PPD$ for the quadratic mapping $x\mapsto x^2+c\bmod{p}$.
We first recall the definition of quadratic residues and quadratic non-residues.
\begin{df}
	Let $p$ be a prime and let $a$ be an integer such that $\gcd(a,p)=1$. We say that $a$ is a \emph{quadratic residue} if the equation
	$x^2= a$ has a solution in $\F_p$. If it doesn't have a solution, then $a$ is called a \emph{quadratic non-residue}.
\end{df}
We have the following well known result for quadratic residues, see \cite{Rosen:2010}:
\begin{thm}\label{thm_quadratic_res}
	Let $p$ be an odd prime. Then there are $(p-1)/2$ quadratic residues and $(p-1)/2$ quadratic non-residues.
\end{thm}

\begin{thm}
	Let $p\geq 3$ be a prime. If $a$ is a quadratic non-residue modulo $p$ then $j= i+a\bmod{p}$ is a desert line.
	If $a$ is a quadratic residue there are two fixed points on the line $j= i+a\bmod{p}$. If $a=0$ there are one fixed point on the line
	$j\equiv i\pmod{p}$.
	Hence, there are exactly $(p-1)/2$ desert lines in $\PPD(h_c,p)$.
\end{thm}
\begin{proof}
	We have an $r$-periodic point at position $(i,j)$ in the $\PPD$-diagram if $h_{i}^r(j)= j\bmod{p}$.
	Consider the points $(i,j)$ on the diagonal line $j= i+a \bmod{p}$.
	Assume that we have an $r$-periodic point on this line. That is, for some $j\in\{0,1,\ldots,p-1\}$ we have
	\[
		h^r_{i}(i+a)= i+a
	\] 
	We have
	\[
		\left(h^{r-1}_{i}(i+a)\right)^2+i= i+a,
	\]
	and
	\begin{equation}\label{eq:quad}
		\left(h^{r-1}_{i}(i+a)\right)^2= a.
	\end{equation}
	Hence, if $a$ is a quadratic non-residue then \eqref{eq:quad} has no solution and $j= i+a \bmod{p}$ is a desert line. From Theorem \ref{thm_quadratic_res} it follows that we have at least $(p-1)/2$ desert lines in the $\PPD$.
	
	Let $r=1$ and let $a$ be a quadratic residue. We have  from \eqref{eq:quad} that $(i+a)^2= a$. There are exactly two values of $i$ that solves this equation since there are exactly two square roots of a quadratic residue. For $a=0$ there is exacty one fixed point (for $i=0$). So, we have exactly $(p-1)/2$ desert lines. 
\end{proof}
\section{The $n$-power case}\label{sec:npower}
In this section we will generalize our investigations from last section to the $\PPD$ of $x\mapsto x^n+c\bmod{p}$.
First we note that if $n$ and $p-1$ are relatively prime then it follows from Theorem \ref{thm_bijective} that we have a pure cycles structure for all values of $c$. This means that we have no desert lines in this case. Moreover, the whole $\PPD$ is filled, since every point is a periodic point. Before we learn more about the desert lines we recall some facts about $n$-power residues.
\begin{df}
	Let $p$ be a prime and let $a\in \F_p^*$. Then $a$ is said to be an $n$-power residue if
	$x^n= a$ has a solution in $\F_p$. Otherwise, $a$ is called an $n$-power non-residue.
\end{df}
\begin{thm}\label{thm:n_power_residues}
	We have that $a \in \F_p$ is an $n$-power residue if and only if
	\[
		a^{(p-1)/\gcd(n,p-1)}= 1.
	\]
	There are $(p-1)/\gcd(n,p-1)$ such $a$ in $\F_p$. For each $n$-power residue $a$ the equation $x^n=a$ has $\gcd(n,p-1)$ solutions in $\F_p$. 
\end{thm}
See for example \cite{Sierpinski:1988} for a proof of this theorem and for more details on $n$-power residues.
\begin{thm}\label{thm:n_power_desert}
	We have that $j= i+a$ is a desert line in $\PPD(h_c,p)$ if and only if $a$ is an $n$-power non-residues modulo $p$. Hence there are exactly $p-1-(p-1)/\gcd(n,p-1)$ desert lines.
	Moreover, if $a$ is an $n$-power residue then there are $\gcd(p-1,n)$ fixed points on the line $j=i+a$. For $a=0$ there is one fixed point (for $i=0$). 
\end{thm}
\begin{proof}
	Consider the points $(i,j)$ on the diagonal line $j=i+a$.
	Assume that we have an $r$-periodic point on this line. That is, for some $i\in\{0,1,\ldots,p-1\}$ we have
	\[
		h^r_{i}(i+a)=i+a.
	\] 
	We have
	\[
		\left(h^{r-1}_{i}(i+a)\right)^n+i= i+a,
	\]
	and
	\begin{equation}\label{eq:npower}
		\left(h^{r-1}_{i}(i+a)\right)^n = a.
	\end{equation}
	Hence, if $a$ is an $n$-power non-residue then $j= i+a$ is a desert line.
	
	If $r=1$ and $a$ is an $n$-power residue then from \eqref{eq:npower} we have $(i+a)^n+i= i+a\pmod{p}$ and hence
	$(i+a)^n= a$. From Theorem \ref{thm:n_power_residues} it follows that there are $\gcd(n,p-1)$ different $i\in \F_p^*$ satisfying this equation.
	For $r=1$ and $a=0$ we have the unique solution $i=0$. Hence there are exactly 
	\[
		p-1-\frac{p-1}{\gcd(n,p-1)}=\frac{p-1}{\gcd(n,p-1)}(\gcd(n,p-1)-1)
	\] 
	desert lines in the $\PPD$. 
\end{proof}
\begin{ex}
	Let us look at the $\PPD$ of the dynamical system $x\mapsto x^5+c\bmod{71}$ in Figure \ref{ppd-71-n5}.
	\begin{figure}	
		\centering
		\includegraphics{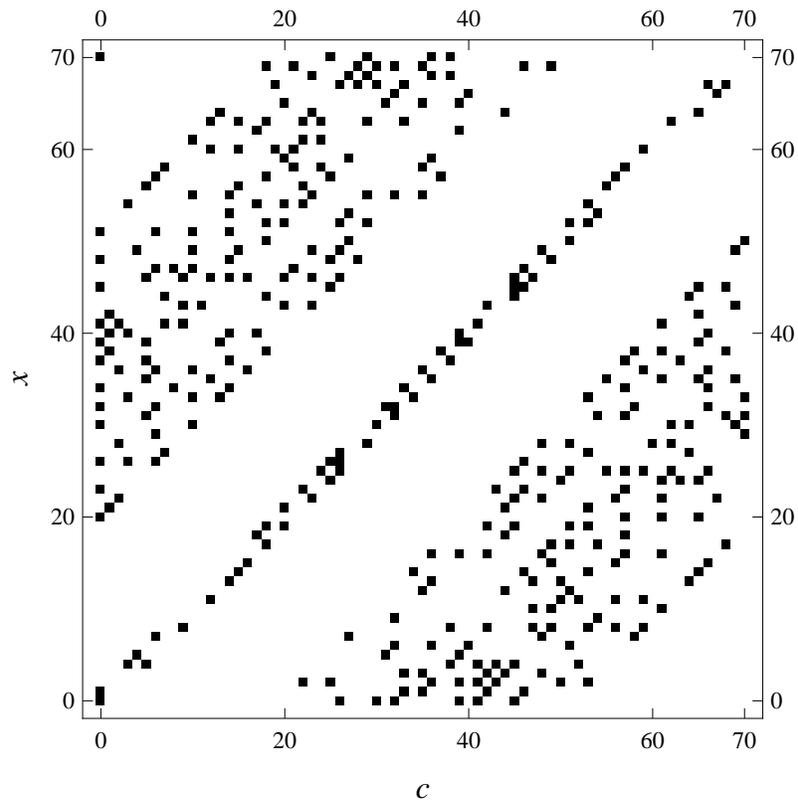}
		\label{ppd-71-n5}
		\caption{The $\PPD(x^5+c,71)$.}
	\end{figure}
	Here we see the increased number of desert lines compared to the $x\mapsto x^2+c\bmod{71}$ in Example \ref{ex:ppd71}. 
\end{ex}
We end this section with an observation regarding symmetry when $n$ is odd.
\begin{df}
	By the \emph{reduced Power Point Diagram} of $h_c(x)$ over $\F_p$ we mean the diagram we get by removing the column $c=0$ and the row $x=0$ from the $\PPD$.
\end{df}
\begin{thm}\label{thm:odd_n_power}
	If $n$ is odd then the reduced $\PPD$ of $h_c$ over $\F_p$ is symmetric with respect to the line $j=-i\bmod{p}$. That is, if $a$ is an $r$-periodic point of $h_c(x)$ then $-a$ is an $r$-periodic points of $h_{-c}(x)$ ($-a$ and $-c$ are the additive inverses of $a$ and $c$ modulo $p$).
\end{thm}
\begin{proof}
	We first observe that  $h_{-c}(-a)=(-a)^n-c= -h_c(a)$ and that
	$h_{-c}^r(-a)= -h_c^r(a)$ for all $r\geq 1$. Assume now that $a$ is an $r$-periodic point of $h_c(x)$. Then
	\[
		h_{-c}^r(-a)= - h_c^r(a)= -a.
	\]
	Moreover, $h_{-c}^d(-a)\neq -a$ for all $d<r$ since $h_c^d(a)\neq a$ for such $d$. Hence $-a$ is an $r$-periodic point of $h_{-c}(x)$.
\end{proof}
\section{Conjectures about the number of periodic points}\label{sec:conjectures}

From Theorem \ref{thm:n_power_desert} it follows that there are exactly $p$ fixed points in total to the set $\{h_c(x):c\in \F_p\}$ of perturbed monomial systems modulo $p$. Hence, among the marked points in the $\PPD$ there are exactly $p$ that are fixed points.

In this section we will consider the problem of computing the number of periodic points for all values of $c$, that is the number of marked points in the $\PPD$. Let $\Per(h_c,p)$ be the number of periodic points of the class of dynamical systems $x\mapsto h_c(x)\bmod{p}$ (the number of marked points in the $\PPD(h_c,p)$). The only case when it is easy to compute $\Per(h_c,p)$ is when $\gcd(n,p-1)=1$. Then we have $\Per(h_c,p)=p^2$. From  \ref{thm:n_power_desert} we have the following rough estimate
\begin{equation}
	p\leq \Per(h_c,p)\leq \frac{p(p-1)}{\gcd(n,p-1)}.
	\label{eqn:per_estimate}
\end{equation}
We have used the fact that there can be at most $p$ periodic points on each non-desert line.
 
We will use a Mathematica program for computing the number of periodic points. The program is based on Floyd's cycle finding algorithm \cite{Floyd:1967,Knuth2:1969}. The naive program described in Section \ref{sec:PPD} for producing the $\PPD$:s is too slow for large primes. Let us look at some examples.
\begin{ex}
 Let us calculate the number of periodic points of $h_c(x)=x^2+c\pmod{p}$ for the first 1000 primes. The result is shown in Figure \ref{fig_1000x2}.
 \begin{figure}
		\includegraphics{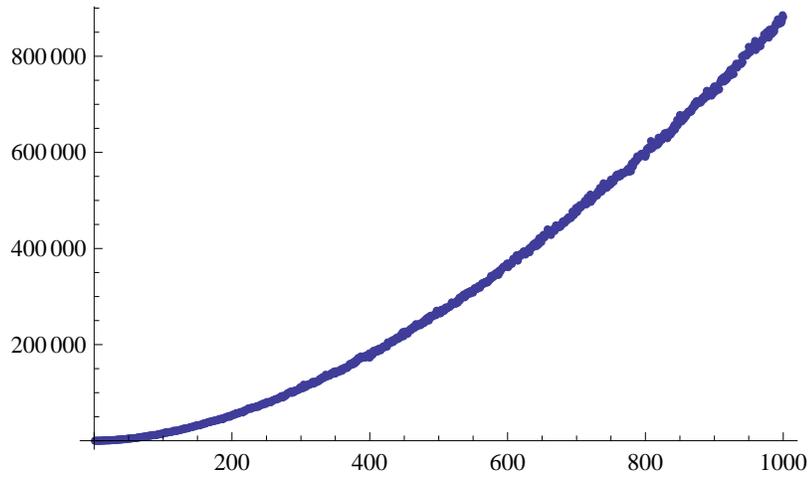}
	\caption{$\Per(x^2+c,p)$ for the first 1000 primes. Note that the value of the horizontal axis is the ordinal of the prime number.}
		\label{fig_1000x2}
\end{figure} 
\end{ex}
\begin{ex}
	Consider the dynamical system $h_c(x)=x^{12}+c$. In Figure \ref{fig1000x12} $\Per(x^{12}+c,p)$ is plotted for the first 1000 primes. Here we see that the diagram seems to contain four different curves. These are branches corresponding to the four possible different values of $\gcd(12,p-1)$ for odd primes $p$. The possible values are $2$, $4$, $6$ and $12$. The lowest branch corresponds to $12$ and the highest to $2$. 
	\begin{figure}
		\includegraphics{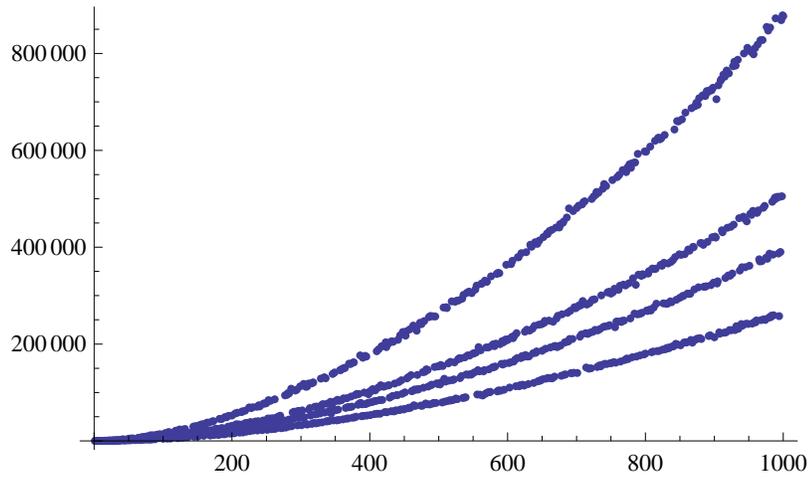}
	\label{fig1000x12}
	\caption{$\Per(x^{12}+c,p)$ for the first 1000 primes.}
\end{figure}
\end{ex}

We call the kind of diagrams shown in the examples \emph{Total Periodic Diagrams} ($\TPD$). The Total Periodic Diagram for $h_c(x)=x^n+c$ for the first $m$ primes is denoted by $\TPD(h_c, m)$. 
\begin{conj}\label{total_conjecture}
	Let $n\geq 2$ be an integer.  Let $h_c(x)=x^n+c\bmod{p}$. The $\TPD(h_c,m)$ then contains as many branches as there are possible values of $\gcd(n,p-1)$.
\end{conj}
It would of course also be interesting to find a simple asymptotic estimate for $\Per(x^n+c,p)$ when $p\to\infty$. Thus, we are interested in finding a function $f_d(p)$ such that
\begin{equation}
	\lim_{\substack{p\to \infty\\\gcd(n,p-1)=d}}\frac{\Per(x^n+c,p)}{f_d(p)}=1,
\end{equation}
for a possible value $d$ of $\gcd(n,p-1)$.
\section{Discussion}
Some of the results of this article can be generalized to arbitrary functions $f(x)$ defined on $\F_p$. Desert lines will occur also in the general case. There will be as many desert lines as there are $a\in \F_p$ such that  $f(x)=a$ has no solutions in $\F_p$. It is also possible to extend some of the results to more general structures than the fields $\F_p$, for example the rings of integers modulo an composite integer. 
\bibliographystyle{plain}
\bibliography{marcusfinite}

\begin{thebibliography}{10}

\bibitem{Floyd:1967}
R.W. Floyd.
\newblock Non-deterministic algorithms.
\newblock {\em J. ACM}, 14(4):636--644, 1967.

\bibitem{Gilbert:2001}
Ch.~L. Gilbert, J.~D. Kolesar, C.~A. Reiter, and J.~D. Storey.
\newblock Function digraphs of quadratic maps modulo $p$.
\newblock {\em Fibonacci Quart.}, 39:32--49, 2001.

\bibitem{Holmgren:1996}
Richard~A. Holmgren.
\newblock {\em A first course in discrete dynamical systems}.
\newblock Springer-Verlag, 1996.

\bibitem{Khrennikov/Nilsson:2001}
A.~Yu. Khrennikov and M.~Nilsson.
\newblock On the number of cycles of $p$-adic dynamical systems.
\newblock {\em Journal of Number Theory}, 90(2):255--264, 2001.

\bibitem{Knuth2:1969}
Donald~E. Knuth.
\newblock {\em The Art of Computer Programming, vol. II: Seminumerical
  Algorithms}.
\newblock Addison-Wesley, 1969.

\bibitem{Nilsson:2003}
M.~Nilsson.
\newblock Fuzzy cycles of $p$-adic monomial dynamical systems.
\newblock {\em Far East J. Dynamical Systems}, 5(2):149--173, 2003.

\bibitem{Nilsson:2007}
M.~Nilsson.
\newblock Computational aspects of monomial dynamical systems.
\newblock {\em The computer journal}, 2007.
\newblock Advanced Access, doi:10.1093/comjnl/bxm100.

\bibitem{Rogers:1996}
T.~D. Rogers.
\newblock The graph of the square mapping on the prime fields.
\newblock {\em Discrete Mathematics}, 148:317--324, 1996.

\bibitem{Rosen:2010}
Kenneth~.H. Rosen.
\newblock {\em Elementary number theory and its applications}.
\newblock Pearson, 2010.

\bibitem{Sierpinski:1988}
W.~Sierpinski.
\newblock {\em Elementary Number theory}.
\newblock North-Holland, 1988.

\bibitem{Vasiga:2004}
T.~Vasiga and J.Shallit.
\newblock On the iteration of certain quadratic maps over gf(p).
\newblock {\em Discrete Mathematics}, 277:219--240, 2004.

\end{thebibliography}
\end{document}